\documentclass[reqno]{amsart}

\usepackage{amssymb,amsmath}
\usepackage{mathrsfs}

\usepackage[thinlines]{easytable}
\usepackage{longtable}
\usepackage{array}
\usepackage{accents}
\usepackage{relsize}
\usepackage{diagbox}

\usepackage{a4wide}
\usepackage{graphicx}
\usepackage{hyperref}
\usepackage[mathscr]{euscript}

\usepackage{mathtools}

\usepackage{subfigure}
\usepackage{calc}
\usepackage{tikz}

\usetikzlibrary{cd,arrows,matrix}
\usepackage{newtxtext}
\usepackage{newtxmath}
\usepackage{anyfontsize}
\usepackage{bm}

\usepackage[english]{babel}

\usepackage[makeroom]{cancel}

\hypersetup{colorlinks=true,linkcolor=blue}

\usepackage{enumitem}
\usepackage{multicol}

\numberwithin{equation}{section}

\theoremstyle{plain}
\newtheorem{theorem}{Theorem}[section]
\newtheorem{lemma}[theorem]{Lemma}
\newtheorem{proposition}[theorem]{Proposition}
\newtheorem{corollary}[theorem]{Corollary}

\newtheorem{definition/proposition}[theorem]{Definition/Proposition}
\newtheorem{theorem/definition}[theorem]{Theorem/Definition}
\newtheorem*{bigtheorem*}{Main Theorem}

\theoremstyle{definition}
\newtheorem{definition}[theorem]{Definition}
\newtheorem{example}[theorem]{Example}
\newtheorem{example/definition}[theorem]{Example/Definition}

\newtheorem{conjecture}[theorem]{Conjecture}
\newtheorem*{conjecture*}{Conjecture}

\theoremstyle{remark}
\newtheorem{remark}[theorem]{Remark}

\makeatletter
\newcommand{\setword}[2]{%
  \phantomsection
  #1\def\@currentlabel{\unexpanded{#1}}\label{#2}%
}
\makeatother

\newcommand{\field}{\mathbf{k}}

\newcommand{\GL}{\mathrm{GL}}

\newcommand{\Gr}{\mathrm{Gr}}

\newcommand{\bZ}{\mathbb{Z}}

\newcommand{\cC}{\mathcal{C}}

\newcommand{\std}{\mathrm{std}}

\newcommand{\modulispace}{\mathcal{M}}

\newcommand{\Betti}{\mathrm{B}}

\newcommand{\Dol}{\mathrm{Dol}}

\newcommand{\NAH}{\mathrm{NAH}}

\newcommand{\sing}{\mathrm{sing}}

\newcommand{\Hilb}{\mathrm{Hilb}}

\newcommand{\cO}{\mathcal{O}}

\newcommand{\coeff}{\mathrm{coeff}}

\def\Linkplus{\tikz[baseline=0.5ex,scale=0.35]{
\draw[->,thick] (1,0) to[out=135,in=-45] (0,1);
\draw[white, fill=white] (0.5,0.5) circle (0.2);
\draw[->,thick] (0,0) to[out=45,in=-135] (1,1);
}
}

\def\Linkminus{\tikz[baseline=0.5ex,scale=0.35]{
\draw[->,thick] (0,0) to[out=45,in=-135] (1,1);
\draw[white, fill=white] (0.5,0.5) circle (0.2);
\draw[->,thick] (1,0) to[out=135,in=-45] (0,1);
}
}

\def\Linkzero{\tikz[baseline=0.5ex,scale=0.35]{
\draw[->,thick] (0,0) to[out=90,in=-90] (0.3,0.5) to[out=90,in=-45] (0,1);
\draw[->,thick] (1,0) to[out=135,in=-90] (0.7,0.5) to[out=90,in=-135] (1,1);
}
}

\newcommand{\tb}{\mathrm{tb}}

\newcommand{\cP}{\mathcal{P}}

\newcommand{\type}{\bm{\mu}}

\newcommand{\HLRV}{\mathbb{H}}

\newcommand{\Br}{\mathrm{Br}}

\newcommand{\rank}{\mathrm{rank}}

\newcommand{\Aug}{\mathrm{Aug}}

\newcommand{\modulistack}{\mathfrak{M}}

\newcommand{\diag}{\mathrm{diag}}

\newcommand{\cE}{\mathcal{E}}
\newcommand{\cF}{\mathcal{F}}

\newcommand{\pa}{\mathrm{par}}

\newcommand{\Res}{\mathrm{Res}}

\newcommand{\pardeg}{\mathrm{pardeg}}

\newcommand{\Spec}{\mathrm{Spec}}

\newcommand{\Eu}{\mathrm{Eu}}

\newcommand{\ms}{\mathrm{s}}

\newcommand{\cW}{\mathcal{W}}

\newcommand{\NR}{\mathrm{NR}}

\begin{document}

\title[Log-concavity from enumerative geometry of planar curve singularities]{Log-concavity from enumerative geometry of planar curve singularities}

\author{Tao Su}
\email{sutao@bimsa.cn}
\address{Beijing Institute of Mathematical Sciences and Applications, Beijing, China}

\author{Baiting Xie}
\email{xbt23@mails.tsinghua.edu.cn}
\address{Tsinghua University, Beijing, China}

\author{Chenglong Yu}
\email{yuchenglong@simis.cn}
\address{Center for Mathematics and Interdisciplinary Sciences, Fudan University and
Shanghai Institute for Mathematics and Interdisciplinary Sciences (SIMIS), Shanghai, China}

\date{}



\begin{abstract}
We propose a log-concavity conjecture for BPS invariants arising in the enumerative geometry of planar curve singularities, identified with the local Euler obstructions of Severi strata in their versal deformations. We further extend this conjecture to ruling polynomials of Legendrian links and to $E$-polynomials of character varieties. We establish these conjectures for irreducible weighted-homogeneous singularities (torus knots) and for ADE singularities, and prove a multiplicative property for ruling polynomials compatible with log-concavity.
\end{abstract}

\maketitle

\tableofcontents

\section*{Introduction}

Log-concavity is a pervasive phenomenon in combinatorics, algebraic geometry, and representation theory. 
A landmark example is the log-concavity of the coefficients of chromatic polynomials of graphs, proved by Huh using Hodge-theoretic methods  \cite{Huh12}; see also \cite{AHK18} for a
far-reaching generalization. More broadly, log-concavity often reflects the presence of hidden geometric structures, such as hard Lefschetz theorems or intersection-theoretic inequalities.

Given a sequence of nonnegative real numbers $a_0,a_1,\cdots,a_n$, recall the following basic notions. The sequence is \textbf{log-concave} if $a_j^2 \ge a_{j-1}a_{j+1}$ for all $1 \le j \le n-1$; it has \textbf{no internal zeros} if there do not exist $i<j<k$ such that $a_i,a_k > 0$ and $a_j = 0$; and it is \textbf{unimodal} if there exists $j$ such that 
\[
a_0 \le a_1 \le \cdots \le a_j \ge a_{j+1} \ge \cdots \ge a_n.
\]
We recall the following elementary properties:
\begin{itemize}[wide,labelwidth=!,labelindent=0pt,itemindent=!]
\item
Any log-concave sequence with no internal zeros is unimodal.

\item
(\textbf{Multiplicativity}) Suppose that 
\[
\sum_{k=0}^K c_k w^k = \left(\sum_{i=0}^I a_i w^i\right)\left(\sum_{j=0}^J b_j w^j\right).
\]
If $(a_i)$ and $(b_j)$ are both log-concave with no internal zeros, then so is $(c_k)$.
\end{itemize}

In this paper, we propose and study a new source of log-concavity arising from enumerative geometry of planar curve singularities, together with several extensions and variations, including ruling polynomials of Legendrian links and $E$-polynomials of character varieties. Throughout the paper, we work over $\field = \mathbb{C}$ unless otherwise specified.

%In particular, these results suggest that log-concavity is a manifestation of a common geometric structure underlying curve singularities, Legendrian topology, and non-abelian Hodge theory.

\subsection*{Main results}
\addtocontents{toc}{\protect\setcounter{tocdepth}{1}}

We summarize the main results of this paper:
\begin{itemize}[wide,labelwidth=!,labelindent=0pt,itemindent=!]
\item we formulate a log-concavity conjecture (Conjecture \ref{conj:log-concavity_for_local_BPS_invariants}) for BPS invariants associated with planar curve singularities;

\item we extend this conjecture to ruling polynomials of Legendrian links (Conjecture \ref{conj:log-concavity_of_ruling_polynomials}) and to $E$-polynomials of character varieties (Conjecture \ref{conj:log-concavity_of_E-polynomials_of_character_varieties});

\item we verify these conjectures for irreducible weighted-homogeneous and ADE singularities (Theorem \ref{thm:log-concavity_for_torus_knots}, Theorem \ref{thm:log-concavity_for_ADE});

\item we establish a multiplicative property for ruling polynomials (Proposition \ref{prop:multiplicativity_for_ruling_polynomials}) compatible with log-concavity.
\end{itemize}

\addtocontents{toc}{\protect\setcounter{tocdepth}{2}}

\section{BPS invariants in the enumerative geometry of planar curve singularities}

Our first object of study is the BPS invariants arising in the enumerative geometry of planar curve singularities, together with their global analogue.
See \cite{She12} for background.

Let $(C,0)$ be a reduced planar curve singularity over $\field$, and let $\delta$ be its $\delta$-invariant.
Let $R=\widehat{\cO}_{C,0}$ be the complete local ring, and let $\widetilde R$ be its normalization. Then
\[
\delta=\delta(C,0)=\dim_{\field}(\widetilde R/R).
\]
Let $\mu=\mu(C,0)$ be the Milnor number, and let $b=b(C,0)$ be the number of local branches of $(C,0)$. Then
\[
\mu=2\delta+1-b.
\]

Let $\Hilb^{[\tau]}=\Hilb^{[\tau]}(C,0)$ be the punctual Hilbert scheme of $\tau$ points on $(C,0)$, namely
\[
\Hilb^{[\tau]}:=\{I\subset R \mid I \text{ is an ideal and } \dim_{\field}(R/I)=\tau\}.
\]
By \cite[Def.~7, Cor.~10]{She12}, there exist unique integers $n_h(C,0)$, $0\leq h\leq \delta$, such that
\begin{equation}\label{eqn:local_BPS_invariants}
q^{-\delta}(1-q)^b \sum_{\tau \ge 0}\chi(\Hilb^{[\tau]})q^{\tau}
=
\sum_{h=0}^{\delta}n_h(C,0)\,z^{2h},
\qquad z:=q^{1/2}-q^{-1/2}.
\end{equation}
This expansion is called the \textbf{genus expansion} of the generating function of Euler characteristics of Hilbert schemes of points on $(C,0)$.
Following \cite[\S6]{She12} and \cite[Appendix~B]{PT10}, we refer to $n_h(C,0)$ as the \textbf{$h$-th BPS invariant} of $(C,0)$.

\begin{conjecture}[Main conjecture]\label{conj:log-concavity_for_local_BPS_invariants}
The sequence $n_0(C,0),\cdots,n_{\delta}(C,0)$ is log-concave with no internal zeros.
In particular, it is unimodal.
\end{conjecture}

\medskip

There is also a global analogue.
Let $C$ be an integral projective locally planar curve over $\field$.
Let $g=g(C)$ and $\tilde{g}=\tilde{g}(C)$ denote its arithmetic and geometric genera, respectively; equivalently, if $\tilde{C}\to C$ is the normalization, then $\tilde{g}$ is the genus of $\tilde{C}$.
Write $\delta(C):=g-\tilde{g}$.

Let $\Hilb^{[\tau]}(C)$ be the Hilbert scheme of $\tau$ points on $C$.
By \cite[Def.~2, Cor.~11]{She12}, there exist unique integers $n_h(C)$\footnote{In Shende's convention, his ``$n_h(C)$'' corresponds to our ``$n_{h-\tilde{g}}(C)$''.}, $0\leq h\leq \delta(C)$, such that
\begin{equation}\label{eqn:global_BPS_invariants}
q^{-g}(1-q)^2 \sum_{\tau \ge 0}\chi(\Hilb^{[\tau]}(C))q^{\tau}
=
\sum_{h=0}^{\delta(C)}n_h(C)\,z^{2(\tilde{g}+h)},
\qquad z:=q^{1/2}-q^{-1/2}.
\end{equation}
We call $n_h(C)$ the \textbf{$h$-th BPS invariant} of $C$.

\begin{conjecture}\label{conj:log-concavity_for_global_BPS_invariants}
The sequence $n_0(C),\cdots,n_{\delta(C)}(C)$ is log-concave with no internal zeros.
In particular, it is unimodal.
\end{conjecture}

\begin{lemma}
Conjecture \ref{conj:log-concavity_for_local_BPS_invariants} holds if and only if Conjecture \ref{conj:log-concavity_for_global_BPS_invariants} holds.
\end{lemma}

\begin{proof}
`$\Leftarrow$'.
Any reduced planar curve singularity $(C,0)$ can be realized as the unique singularity of an integral rational projective curve $C$.
Then by \cite[Prop.~8]{She12}, $n_h(C)=n_h(C,0)$.

\smallskip

`$\Rightarrow$'.
Let $\{c_i\}_{i\in I}$ be the singular points of $C$.
Again by \cite[Prop.~8]{She12}, we have
\[
n_h(C)=\sum_{\sum_{i\in I}h_i=h}\prod_{i\in I} n_{h_i}(C,c_i).
\]
The claim then follows from the multiplicativity of log-concavity.
\end{proof}

\medskip

There is a further geometric interpretation of the BPS invariants.
Let $\pi:\cC\to S$ be a projective flat family of integral locally planar curves over a smooth base such that $C=\cC_0$.
Suppose that $\pi$ is versal at $0\in S$, i.e., $\pi:(\cC,\sing(\cC_0))\to (S,0)$ is a versal deformation of $(C,\sing(C))$, and let $\pi:\cC\to S$ be a small representative.

The \textbf{global Severi strata} are defined by
\[
S_h:=\overline{\{s\in S \mid \cC_s \text{ is nodal with } \delta(C)-h \text{ nodes}\}}.
\]
By \cite[Thm.~47]{GS14}, we have
\[
n_h(C)=\Eu_{S_h}(0),
\]
where $\Eu_{S_h}(0)$ denotes the local Euler obstruction at $0$.
In particular, if $C$ is rational and $(C,0)$ is its unique singularity, then
\[
n_h(C,0)=\Eu_{V_h}(0),
\]
where $V_h = V_h(C,0)$ denotes the $h$-th \textbf{local Severi stratum} in a versal deformation $\cC \to V$ of $(C,0)$, i.e.,
\[
V_h := \overline{\{v \mid \cC_v \text{ is nodal with } \delta-h \text{ nodes}\}}.
\]

\begin{remark}
It is known that $S_h$ has pure codimension $\delta(C)-h$.
Its multiplicity at $0$, denoted $\deg_0(S_h)$, is called the $h$-th global Severi degree.

We thank V.~Shende for kindly pointing out an issue\footnote{More precisely, \cite[Lem.~18]{She12} is not correct as stated.} in the proof of the main result of \cite{She12}, and therefore we avoid using the identity $n_h(C)=\deg_0(S_h)$.
To the best of our knowledge, no counterexample to this identity is currently known.

Nevertheless, one still has the following multiplicativity property.
Let $\{c_i\}_{i\in I}$ be the singular points of $C$, and let $V(c_i)$ be the base of the miniversal deformation of $(C,c_i)$ with local Severi strata $V_{h_i}(c_i)$.
Then there exists a smooth morphism
\[
p:S\to\prod_{i\in I}V(c_i)
\]
such that
\[
S_h=p^{-1}\!\left(\prod_{h=\sum h_i}V_{h_i}(c_i)\right),
\]
and hence
\[
\deg_0(S_h)=\sum_{h=\sum h_i}\prod_{i\in I}\deg_0\bigl(V_{h_i}(c_i)\bigr).
\]
See \cite[p.~532]{She12}.
This is again compatible with the multiplicativity of log-concavity.
\end{remark}

\section{Ruling polynomials of Legendrian links}

Let $\Lambda \hookrightarrow (\mathbb{R}^3_{x,y,z},\alpha_{\std} = dz - ydx)$ be an oriented Legendrian link in the standard contact three-space.
Let $\pi_{xz}:\mathbb{R}^3_{x,y,z} \to \mathbb{R}^2_{x,z}$ denote the front projection.
Let $R_{\Lambda}(z) \in \mathbb{Z}_{\geq 0}[z^{\pm 1}]$ be the \textbf{$\mathbb{Z}/2$-graded ruling polynomial} of $\Lambda$ \cite[Def.~2.3]{HR15}.

Define the \textbf{normalized $\mathbb{Z}/2$-graded ruling polynomial} by
\[
\widetilde{R}_{\Lambda}(z):= z^b R_{\Lambda}(z),
\]
where $b := |\pi_0(\Lambda)|$ is the number of connected components of $\Lambda$.
It is known that $\widetilde{R}_{\Lambda}(z)\in \mathbb{Z}_{\geq 0}[z^2]$.

There is an alternative description of $R_{\Lambda}(z)$ when $\Lambda = \beta^{>}$ is the rainbow closure \cite[\S 6.1]{STZ17} of a positive braid $\beta$.

Write $\beta = \sigma_{i_1}\cdots \sigma_{i_{\ell}} \in \Br_n^+$, 
where $\sigma_{i_1}$ is applied first (i.e., the braid is read from left to right).
We label the $n$ parallel strands from bottom to top by $1,2,\ldots,n$. 
Let $W = S_n$ be the symmetric group on $[n] = \{1,\ldots,n\}$, and let $w_0 \in W$ be the longest element, i.e., $w_0(i) = n+1-i$. 
For $1 \le k \le n-1$, let $\ms_k := (k~k+1) \in W$, and define
\[
\ms_{\beta} := \ms_{i_{\ell}} \circ \cdots \circ \ms_{i_1} \in W.
\]
Then $|\pi_0(\beta^{>})|$ is equal to the number of cycles (including $1$-cycles) of $\ms_{\beta}$.

\begin{definition}[{\cite[\S 5.4]{Mel25}}]\label{def:walks}
Let $p=(p_{\ell},\ldots,p_0)\in W^{\ell+1}$ with $p_0 = p_{\ell} = w_0$. 
We say that $p$ is a \emph{$w_0$-$w_0$-walk} of $\beta$ if for each $1 \le m \le \ell$,
\[
p_m =
\begin{cases}
\ms_{i_m} p_{m-1} & \text{if } \ms_{i_m} p_{m-1} > p_{m-1},\\
\ms_{i_m} p_{m-1} \ \text{or}\ p_{m-1} & \text{if } \ms_{i_m} p_{m-1} < p_{m-1},
\end{cases}
\]
where $>$ denotes the Bruhat order on $W$.
Let $S_p := \{\, m \mid p_m = p_{m-1} \,\}$, and denote by $\cW(\beta)$ the set of all such walks. 
\end{definition}

\begin{remark}\label{rem:walks}
Let $w_0 = \ms_{j_{\binom{n}{2}}}\circ \cdots \circ \ms_{j_1}$. Then 
\[
\Delta := \sigma_{j_{\binom{n}{2}}}\circ \cdots \circ \sigma_{j_1} \in \Br_n^+
\]
is the half-twist lifting $w_0$. Set $\tilde{\beta} := \Delta\beta\Delta$. So, $\ell(\tilde{\beta}) = \ell + n(n-1)$.

A direct check shows that a $w_0$-$w_0$-walk $p$ of $\beta$ corresponds to a walk $\tilde{p}$ of $\Delta\beta\Delta$ in \cite[Def.~3.3]{Su25} via:
\[
\tilde{p}_m := \ms_{j_m}\cdots\ms_{j_1} \quad (0 \le m \le \tbinom{n}{2}),\qquad
\tilde{p}_{\binom{n}{2}+m} := p_m \quad (0 \le m \le \ell),
\]
\[
\tilde{p}_{\binom{n}{2}+\ell+m} := \ms_{j_m}\cdots\ms_{j_1} w_0 \quad (0 \le m \le \tbinom{n}{2}).
\]
See also Section \ref{subsec:BPS_invariants_vs_E-polynomials} for a geometric interpretation of $\cW(\beta)$.
\end{remark}

There is a canonical bijection between $\cW(\beta)$ and the set $\NR(\beta^{>})$ of ($\mathbb{Z}/2$-graded) normal rulings of $\beta^{>}$. 
To describe it, label the upper $n$ strands of $\beta^{>}$ (i.e., the strands at the top of $\beta$) from bottom to top by $n+1,\ldots,2n$. 
Given $p=(p_{\ell},\ldots,p_0)\in \cW(\beta)$, the corresponding ruling $\rho$ is determined as follows: 
at each position $0 \le m \le \ell$, the strand $n+j$ is paired with $p_m(j)$.\footnote{
If $0<m\le \ell$, then $m$ denotes the vertical line immediately to the right of $\sigma_{i_m}$; if $m=0$, it denotes the vertical line immediately to the left of $\sigma_{i_1}$.}

Moreover, for each $1 \le m \le \ell$, we have $m \in S_p$ if and only if the crossing $\sigma_{i_m}$ is a switch of $\rho$. 
In particular, $|S_p| = |S(\rho)|$, and hence
\begin{equation}\label{eqn:ruling_polynomial_of_rainbow_closure_of_positive_braid}
R_{\beta^{>}}(z) = \sum_{p \in \cW(\beta)} z^{|S_p| - n}.
\end{equation}

\begin{conjecture}\label{conj:log-concavity_of_ruling_polynomials}
Let $\Lambda = \beta^{>}$ be the rainbow closure \cite[\S 6.1]{STZ17} of a positive braid $\beta$.
Write
\[
\widetilde{R}_{\Lambda}(z) = \sum_{j=0}^{\delta} a_j z^{2j}.
\]
Then the sequence $(a_0,\ldots,a_{\delta})$ is log-concave with no internal zeros.
In particular, it is unimodal.
\end{conjecture}

\begin{example}\label{ex:computation_of_ruling_polynomial}
Let $\Lambda = (\sigma_1^2\sigma_2^2\sigma_3^2\sigma_4^2\sigma_3^2\sigma_2\sigma_1)^{>}$, then $b = |\pi_0(\Lambda)| = 3$, and by a direct computation we have
\[
\widetilde{R}_{\Lambda}(z) = 4 + 20z^2 + 33z^4 + 24z^6 + 8z^8 + z^{10} = (1+z^2)(2+z^2)^2(1+3z^2+z^4),
\]
which indeed satisfies Conjecture \ref{conj:log-concavity_of_ruling_polynomials}.
For a simpler computation, see also Proposition \ref{prop:multiplicativity_for_ruling_polynomials}.
\end{example}

Our original motivation comes from the Legendrian knot atlas of Chongchitmate--Ng \cite{CN24}, where the normalized ruling polynomials $\widetilde{R}_{\Lambda}(z)=z R_{\Lambda}(z)$ are listed for many Legendrian knots ($b = 1$). 
A direct check shows that all examples in the atlas satisfy Conjecture \ref{conj:log-concavity_of_ruling_polynomials}.

\begin{remark}
We are very grateful to Yu Pan for kindly sharing with us their results \cite{CSP26}, which state that any polynomial in $z^2$ with nonnegative integer coefficients can be realized as the normalized ($\mathbb{Z}$-graded) ruling polynomial of a Legendrian link. 
Thus, Conjecture~\ref{conj:log-concavity_of_ruling_polynomials} does not extend to arbitrary Legendrian links.

Nevertheless, our primary interest lies in Legendrian links associated with planar curve singularities, for which the conjecture does not appear to contradict their results. 
Moreover, extensive computational evidence supports Conjecture~\ref{conj:log-concavity_of_ruling_polynomials}; for example, it holds for all $\beta \in \Br_n^+$ with $n \le 6$ and $\ell(\beta) \le 15$.
\end{remark}

\begin{remark}\label{rem:ruling_polynomial_vs_HOMFLY-PT}
Let $P(a,z)$ denote the \textbf{HOMFLY--PT polynomial} of oriented links, defined by
\begin{equation}
a^{-1} P(\Linkplus) - a P(\Linkminus) = z P(\Linkzero), \qquad P(\text{unknot}) = 1.
\end{equation}
By \cite[Thm.~4.3]{Rut06}\footnote{In Rutherford's convention for the HOMFLY--PT polynomial, his ``$a$'' corresponds to our ``$a^{-1}$''.}, the lowest $a$-degree of $P_{\Lambda}(a,z)$ (with nonzero coefficient) is at least $\tb(\Lambda)+1$, and
\begin{equation}
z R_{\Lambda}(z) = \coeff_{a^{\tb(\Lambda)+1}} \bigl(P_{\Lambda}(a,z)\bigr).
\end{equation}
Here $\tb(\Lambda)$ denotes the Thurston--Bennequin number of $\Lambda$.
\end{remark}

\section{\texorpdfstring{$E$}{E}-polynomials of character varieties}

Let $\Sigma_{g,k}$ be a Riemann surface of genus $g$ with $k>0$ punctures such that $2-2g-k<0$.
Let $G=\GL_n(\field)$, and let $T\subset G$ be the diagonal maximal torus.
Fix $C_i\in T$ for $1\le i\le k$ such that $\prod_{i=1}^k \det C_i = 1$.

Define $\modulispace_{\Betti}$ to be the $G$-character variety of $\Sigma_{g,k}$ with local monodromy at the $i$-th puncture conjugate to $C_i$. More precisely,
\[
\modulispace_{\Betti}
:= \left\{ A_j,B_j\in G,\; M_i\in G\cdot C_i \ \middle|\ 
\prod_{j=1}^g (A_j B_j A_j^{-1} B_j^{-1}) \cdot \prod_{i=1}^k M_i = I_n
\right\} \bigg/ \!\!\bigg/ G,
\]
where $G\cdot C_i$ denotes the conjugacy class of $C_i$, and $//$ denotes the affine GIT quotient.

By work of Hausel--Letellier--Rodriguez-Villegas \cite{HLRV11,HLRV13}, under a generic condition, $\modulispace_{\Betti}$ is a smooth connected affine variety.
By results of Shende \cite{She17} and Mellit \cite{Mel25}, its compactly supported cohomology $H_c^*(\modulispace_{\Betti})$ is of Hodge--Tate type. 
Thus the mixed Hodge structure is determined by the weight filtration.

The $E$-polynomial of $\modulispace_{\Betti}$ is defined by
\[
E(\modulispace_{\Betti};q)
= \sum_{i,j} \dim \Gr^W_{2i} H_c^j(\modulispace_{\Betti})\, q^i (-1)^j.
\]
A theorem of Katz \cite[Appendix]{HRV08} identifies this polynomial with the point count of $\modulispace_{\Betti}$ over finite fields. 
Moreover, Mellit \cite{Mel25} showed that $\modulispace_{\Betti}$ satisfies the curious hard Lefschetz property, which implies that $E(\modulispace_{\Betti};q)$ is palindromic.

As a consequence, there exists a unique polynomial
\[
\widetilde R(\modulispace_{\Betti};z)\in \bZ[z^2]
\]
such that
\[
q^{-\frac{d}{2}} E(\modulispace_{\Betti};q)
= \widetilde R(\modulispace_{\Betti};z), 
\qquad z = q^{1/2} - q^{-1/2},
\]
where $d = \dim_{\field} \modulispace_{\Betti}$.
We refer to this expansion as the \textbf{genus expansion} of the character variety.

\begin{conjecture}\label{conj:log-concavity_of_E-polynomials_of_character_varieties}
Write
\[
\widetilde{R}(\modulispace_{\Betti};z)
= \sum_{h=0}^{d/2} a_h z^{2h}.
\]
Then the sequence $(a_0,\cdots,a_{d/2})$ is nonnegative, log-concave, and has no internal zeros.
In particular, it is unimodal.
\end{conjecture}

%\noindent\textbf{Note.}

\begin{remark}
For simplicity, we state the conjecture only for tame $GL_n(\field)$-character varieties. 
We expect that it admits a natural extension to smooth wild character varieties. 
Moreover, the conjecture can be formulated for any complex reductive algebraic group in place of $GL_n(\field)$.
\end{remark}

\begin{remark}\label{rem:E-polynomial_vs_modified_Macdonald_symmetric_functions}
Let $\mu^i\in\cP_n$ be the partition encoding the eigenvalue multiplicities of $C_i\in T$, and write $\type=(\mu^1,\cdots,\mu^k)$.
Then by \cite{HLRV11},
\[
E(\modulispace_{\Betti};q)
= q^{\frac{d}{2}} \HLRV_{\type}(q^{1/2},q^{-1/2}),
\]
where $\HLRV_{\type}$ is the \textbf{HLRV function} defined via modified Macdonald symmetric functions.
\end{remark}

Thanks to Remark \ref{rem:E-polynomial_vs_modified_Macdonald_symmetric_functions}, we can compute $E(\modulispace_{\Betti};q)$ in many examples.
For instance, let $(g,k,n)=(2,2,3)$ and let $C_1,C_2\in T \cong (\field^\times)^3$ be regular semisimple.
Then $d=\dim \modulispace_{\Betti}=32$, and
\begin{eqnarray*}
\widetilde{R}(\modulispace_{\Betti};z)
&=& q^{-\frac{d}{2}} E(\modulispace_{\Betti};q) \\
&=& z^{32} + 32z^{30} + 460z^{28} + 3916z^{26} + 21902z^{24} + 84340z^{22} \\
&& + 227630z^{20} + 429340z^{18} + 552720z^{16} + 461160z^{14} \\
&& + 225000z^{12} + 51120z^{10} + 2640z^8.
\end{eqnarray*}
A direct check shows that Conjecture \ref{conj:log-concavity_of_E-polynomials_of_character_varieties} holds in this example.

\section{The interconnections}

At first sight, the three conjectures above appear to be unrelated. 
However, we explain below that, when combined with known results, Conjecture~\ref{conj:log-concavity_of_ruling_polynomials} and Conjecture~\ref{conj:log-concavity_of_E-polynomials_of_character_varieties} can be viewed as generalizations of Conjecture~\ref{conj:log-concavity_for_local_BPS_invariants}.

\subsection{BPS invariants vs.\ ruling polynomials}

Let $(C,0)$ be a reduced planar curve singularity as in Conjecture~\ref{conj:log-concavity_for_local_BPS_invariants}. 
Let $R=\widehat{\cO}_{C,0}$ be the complete local ring, and let $\widetilde R$ be its normalization. Then
\[
\delta=\delta(C,0)=\dim_{\field}(\widetilde R/R).
\]
Let $\mu=\mu(C,0)$ be the Milnor number, and let $b=b(C,0)$ be the number of local branches of $(C,0)$. Then
\[
\mu=2\delta+1-b.
\]
By Maulik's proof \cite{Mau16} of the Oblomkov--Shende conjecture \cite[Conj.~1, Conj.~2']{OS12}, we have
\begin{equation}\label{eq:Hilbert_scheme_q-function_vs_lowest_a-coefficient_of_HOMFLY-PT}
q^{-\mu/2}(1-q)\sum_{\tau\ge 0}\chi(\Hilb^{[\tau]})q^{\tau}
=
\coeff_{a^{\mu}}\bigl(P_{L_{(C,0)}}(a,z)\bigr),
\end{equation}
where $P_{L_{(C,0)}}(a,z)$ is the HOMFLY--PT polynomial of the singularity link $L_{(C,0)}$.

Define
\[
\widetilde R_{L_{(C,0)}}(z)
:=
z^{b-1}\cdot \bigl(\text{lowest $a$-coefficient of } P_{L_{(C,0)}}(a,z)\bigr).
\]
Combining the definition (\ref{eqn:local_BPS_invariants}) of BPS invariants and \eqref{eq:Hilbert_scheme_q-function_vs_lowest_a-coefficient_of_HOMFLY-PT}, we obtain
\begin{equation}\label{eqn:BPS_invariants_vs_lowest_HOMFLY-PT}
\sum_{h=0}^{\delta} n_h(C,0) \,z^{2h}
=
\widetilde R_{L_{(C,0)}}(z).
\end{equation}

Next, as an algebraic link, $L_{(C,0)}$ can be represented as the closure of a positive braid. 
Write
\[
L_{(C,0)}=\beta^{\circ}
\qquad (\beta\in \Br_n^+).
\]
By \cite[Thm.~2]{Sta78}, $L_{(C,0)}$ is a fibered link whose oriented fiber surface $T_{\beta}$ is obtained as the union of $n$ disks, one for each strand, with the $i$-th and $(i+1)$-st disks joined by one half-twisted strip for each occurrence of $\sigma_i$ in $\beta$.
Since the associated fibration\footnote{\textbf{Note}. Such a fibration corresponds to a primitive class in $H^1(S^3\setminus L_{(C,0)},\mathbb{Z}) = \mathbb{Z}^b$.}
\[
S^3\setminus L_{(C,0)} \to S^1
\]
is unique up to isotopy, so is the fiber surface. By the Milnor fibration theorem \cite[Thm.~4.8, Thm.~7.2]{Mil68}, it follows that $T_{\beta}$ is the Milnor fiber of the planar curve singularity $(C,0)$, and hence
\[
\mu=\rank H_1(T_{\beta})=1-\chi(T_{\beta})=1-n+e(\beta),
\]
where $e(\beta)$ is the number of crossings of $\beta$.

Now let
\[
\Lambda:=\beta^{>}
\]
be the rainbow closure of $\beta$ \cite[\S 6.1]{STZ17}, viewed as a Legendrian link in the standard contact three-space $(\mathbb{R}^3_{x,y,z},\alpha_{\std}=dz-ydx)$.
Then
\[
\tb(\Lambda)=e(\beta)-n,
\]
and therefore
\[
\mu=1+\tb(\Lambda).
\]
By the definition of $\widetilde R_{L_{(C,0)}}(z)$ and Remark~\ref{rem:ruling_polynomial_vs_HOMFLY-PT}, we conclude that
\begin{equation}\label{eqn:BPS_invariants_vs_ruling_polynomials}
\widetilde R_{L_{(C,0)}}(z)=\widetilde R_{\Lambda}(z).
\end{equation}

In summary, we obtain the following implication.

\begin{lemma}
Conjecture~\ref{conj:log-concavity_of_ruling_polynomials} implies Conjecture~\ref{conj:log-concavity_for_local_BPS_invariants}.
\end{lemma}

\subsection{BPS invariants vs.\ \texorpdfstring{$E$}{E}-polynomials}\label{subsec:BPS_invariants_vs_E-polynomials}

Let $(C,0)$, $\beta\in \Br_n^+$, and $\Lambda=\beta^{>}$ be as above. Recall that $T\subset \GL_n(\field)$ is the standard diagonal torus.
By \cite[Thm.~0.6, Rmk.~3.7]{Su25}, we have
\begin{equation}\label{eq:augmentation_variety_vs_wild_character_stack}
[\Aug(\beta^{>},*_1,\cdots,*_n)/T]
\cong
\modulistack_1(\mathbb{P}^1,\{\infty\},(\Delta\beta\Delta)^{\circ})
\cong
[X(\Delta\beta,w_0)/T],
\end{equation}
where:
\begin{itemize}[wide,labelwidth=!,labelindent=0pt,itemindent=!]
\item
$\Aug(\beta^{>},*_1,\cdots,*_n)$ is the augmentation variety associated to the Legendrian link $\beta^{>}$, where $*_i$ is a base point placed at the $i$-th innermost right cusp of $\beta^{>}$. See \cite{Su17} for details;

\item
$X(\Delta\beta,w_0)$ is the modified braid variety \cite[(3.1)]{Su25} associated to $(\beta,w_0)$, where $w_0$ is the longest element of $S_n$. In fact, there is a natural $T$-equivariant isomorphism
\[
\Aug(\beta^{>},*_1,\cdots,*_n)\cong X(\Delta\beta,w_0);
\]

\item
$\modulistack_1(\mathbb{P}^1,\{\infty\},(\Delta\beta\Delta)^{\circ})$ is the wild character stack on $\mathbb{P}^1$ with one irregular singularity at $\infty$, specified by the Stokes Legendrian link $(\Delta\beta\Delta)^{\circ}$.
More concretely, the Legendrian link lives in the cosphere bundle $T^{\infty}(\mathbb{P}^1\setminus\{\infty\})$, and is identified with its front diagram $(\Delta\beta\Delta)^{\circ}$ encircling $\infty$, where $\Delta$ is the half-twist.
The stack $\modulistack_1(\mathbb{P}^1,\{\infty\},(\Delta\beta\Delta)^{\circ})$ is the moduli stack of microlocal rank-one constructible sheaves on $\mathbb{P}^1$ with acyclic stalk at $\infty$, whose microsupport at contact infinity is contained in $(\Delta\beta\Delta)^{\circ}$.
\end{itemize}

Let
\[
\modulistack_{\Betti}(\beta):=
\modulistack_1(\mathbb{P}^1,\{\infty\},(\Delta\beta\Delta)^{\circ}),
\]
and let
\[
d_{\beta}:=\dim \Aug(\beta^{>},*_1,\cdots,*_n)-n
=\dim X(\Delta\beta,w_0)-n.
\]
Recall that $X(\Delta\beta,w_0)$ admits a cell decomposition (see \cite[Prop.~3.4 or Page~33]{Su25} and Remark \ref{rem:walks}):
\begin{equation}\label{eqn:cell_decomposition_of_modified_braid_varieties}
X(\Delta\beta,w_0) = \bigsqcup_{p\in \cW(\beta)} X_p(\Delta\beta,w_0),\qquad 
X_p(\Delta\beta,w_0) \cong (\field^{\times})^{|S_p|}\times \field^{|U_p|},
\end{equation}
where $|S_p| + 2|U_p| = \ell(\beta)$. In particular, it follows that
\[
d_{\beta}=\tb(\beta^{>}) = \ell(\beta) - n.
\]

Define
\[
R(\modulistack_{\Betti}(\beta);z)
:=
q^{-d_{\beta}/2}\,|\modulistack_{\Betti}(\beta;\mathbb{F}_q)|
=
q^{-d_{\beta}/2}\frac{|\Aug(\beta^{>},*_1,\cdots,*_n;\mathbb{F}_q)|}{(q-1)^n}
=
q^{-d_{\beta}/2}\frac{|X(\Delta\beta,w_0;\mathbb{F}_q)|}{(q-1)^n}.
\]
Then
\[
\widetilde R(\modulistack_{\Betti}(\beta);z):=z^bR(\modulistack_{\Betti}(\beta);z)
\]
is a wild analogue of $\widetilde R(\modulispace_{\Betti};z)$ from Conjecture~\ref{conj:log-concavity_of_E-polynomials_of_character_varieties}. Moreover, we have
\begin{equation}\label{eq:E-polynomial_vs_ruling_polynomial}
\widetilde R(\modulistack_{\Betti}(\beta);z)=\widetilde R_{\beta^{>}}(z).
\end{equation}
Indeed, using the cell decomposition \eqref{eqn:cell_decomposition_of_modified_braid_varieties} and \eqref{eqn:ruling_polynomial_of_rainbow_closure_of_positive_braid}, a direct computation shows that
\[
R(\modulistack_{\Betti}(\beta);z) = \sum_{p\in\cW(\beta)} z^{|S_p|-n} = R_{\beta^>}(z).
\]
Thus, \eqref{eq:E-polynomial_vs_ruling_polynomial} follows.

Combined with \eqref{eqn:BPS_invariants_vs_ruling_polynomials}, \eqref{eq:E-polynomial_vs_ruling_polynomial} gives the following implication.

\begin{lemma}
Conjecture~\ref{conj:log-concavity_of_E-polynomials_of_character_varieties} in the wild case implies Conjecture~\ref{conj:log-concavity_for_local_BPS_invariants}.
\end{lemma}

%\medskip

\subsection{A heuristic relation with the BPS picture}

It is also worth mentioning a more intuitive relationship between character varieties and the BPS picture.

Start with a generic (hence smooth) character variety $\modulispace_{\Betti}$ of dimension $d$ on $\Sigma_{g,k}$ whose local monodromy at the $j$-th puncture is conjugate to
\[
C_j=\diag(e^{2\pi i a_{j,1}}I_{\mu^j_1},\cdots,e^{2\pi i a_{j,r_j}}I_{\mu^j_{r_j}})\in T,
\]
where
\[
0\le a_{j,1}<a_{j,2}<\cdots<a_{j,r_j}<1,
\qquad
\sum_i \mu_i^j=n.
\]
By nonabelian Hodge theory for punctured curves \cite{Sim90,Sim92,Sim94,Kon93}, there is a diffeomorphism
\[
\NAH:\modulispace_{\Dol}\simeq \modulispace_{\Betti},
\]
where $\modulispace_{\Dol}$ is the moduli space of stable parabolic regular $G$-Higgs bundles
\[
\cE_{\pa}=(\cE,\theta,\{\cF_j\}_{j=1}^k)
\]
on $(\Sigma_g,D=p_1+\cdots+p_k)$ of parabolic degree zero. Here:
\begin{itemize}[wide,labelwidth=!,labelindent=0pt,itemindent=!]
\item
$\cE$ is a rank-$n$ holomorphic vector bundle on $\Sigma_g$;

\item
$\theta:\cE\to \cE\otimes K_{\Sigma_g}(D)$ is an $\cO_{\Sigma_g}$-linear Higgs field;

\item
for each $j$,
\[
\cE_{p_j}=\cF_{j,1}\supset \cdots \supset \cF_{j,r_j+1}=0,
\qquad
\rank(\cF_{j,i}/\cF_{j,i+1})=\mu_i^j;
\]

\item
the residue satisfies
\[
\Res_{p_j}(\theta)\,\cF_{j,i}\subset \cF_{j,i+1};
\]

\item
the parabolic degree is
\[
\pardeg(\cE_{\pa})
=
\deg(\cE)+\sum_{j=1}^k\sum_{i=1}^{r_j} a_{j,i}\mu_i^j.
\]
\end{itemize}

Thus \(\NAH\) induces an isomorphism
\[
\NAH^*:H^*(\modulispace_{\Betti})\cong H^*(\modulispace_{\Dol}).
\]
By the \(P=W\) conjecture \cite{dCHM12}, now proved in \cite{HMMS22,MS24,MSY25} and expected to extend to the present setting, one has
\[
P_{\bullet}H^*(\modulispace_{\Dol})\cong W_{2\bullet}H^*(\modulispace_{\Betti}),
\]
where \(P_{\bullet}\) is the perverse filtration with respect to the Hitchin fibration
\[
h:\modulispace_{\Dol}\to \mathbb{A},
\]
and \(W_{\bullet}\) is the weight filtration. For a partial geometric interpretation, see \cite{Sim16, KNPS15, Su23}.
Denote
\[
P(\modulispace_{\Dol};q,t)
:=
\sum_{i,j}\dim \Gr_i^P H^j(\modulispace_{\Dol})\,q^i t^j.
\]
Then
\[
E(\modulispace_{\Betti};q)=P(\modulispace_{\Dol};q,-1),
\]
which may be viewed as a perverse \(E\)-polynomial.

By a suitable generalization of the spectral correspondence \cite{BNR89}, the Hitchin fibration can be interpreted as a ``relative compactified Jacobian''
\[
\overline{J}\cC_{\Dol}\to \mathbb{A},
\]
where
\[
\pi_{\Dol}:\cC_{\Dol}\to \mathbb{A}
\]
is the family of ``spectral curves''.
In particular, if the spectral curve \(C_a:=\cC_{\Dol,a}\) is integral, then the Hitchin fiber
\[
h^{-1}(a)=\overline{J}C_a
\]
is the compactified Jacobian of \(C_a\).

The nilpotent residue condition yields a \(\mathbb{G}_m\)-action on \(\modulispace_{\Dol}\) scaling the Higgs field. It follows that
\[
P_{\bullet}H^*(\modulispace_{\Dol})
\cong
P_{\bullet}H^*(h^{-1}(0)).
\]
Ignoring for the moment the nonreduced issues of the central spectral curve, one is led to consider an abstract integral curve \(C=C_0\). By the Macdonald formula for \textbf{integral} locally planar curves \cite{MS13,MY14},
\[
H^k(\Hilb^{[\tau]}(C))
\cong
\bigoplus_{i+j\le \tau,\ i,j\ge 0}
\Gr_i^P\bigl(H^{k-2j}(\overline{J}C)\bigr)(-j),
\]
where \(\Hilb^{[\tau]}(C)\) is the Hilbert scheme of \(\tau\) points on \(C\).
Hence
\[
\sum_{k\ge 0,\tau\ge 0}\dim H^k(\Hilb^{[\tau]}(C))\,q^{\tau}t^k
=
\frac{P(\modulispace_{\Dol};q,t)}{(1-q)(1-t^2q)}.
\]
In particular, at \(t=-1\), using the definition (\ref{eqn:global_BPS_invariants}) of BPS invariants,
\begin{equation}\label{eq:perverse_E-polynomial_vs_Hilbert_scheme_q-function}
P(\overline{J}C;q,-1)
=
(1-q)^2\sum_{\tau\ge 0}\chi(\Hilb^{[\tau]}(C))q^{\tau} = q^{g(C)}\sum_{h=0}^{\delta(C)}n_h(C)z^{2(\tilde{g}(C)+h)}.
\end{equation}
Since
\[
d=\dim \modulispace_{\Dol}=2\dim h^{-1}(0)=2g(C_0),
\]
this suggests the heuristic identity
\[
\widetilde R(\modulispace_{\Betti};z)
=
q^{-d/2}P(\modulispace_{\Dol};q,-1)
=
\sum_{h=0}^{\delta(C_0)} n_h(C_0)\,z^{2(\tilde g(C_0)+h)}.
\]
Here \(n_h(C_0)\) is interpreted as the \(h\)-th BPS invariant of \(C_0\), while the central ``spectral curve'' \(C_0\) may be highly non-reduced. Nevertheless, this provides a conceptual bridge between Conjecture~\ref{conj:log-concavity_of_E-polynomials_of_character_varieties} and Conjecture~\ref{conj:log-concavity_for_global_BPS_invariants}.

\subsection{A wild case}

In fact, the above discussion becomes more precise in a wild case.
Let \(C\) be a rational integral projective curve with a unique irreducible planar curve singularity \((C,0)\). Then
\[
b=1,\qquad \tilde g(C)=0,\qquad \dim \overline{J}C=g(C)=\delta.
\]
Let
\[
L_{(C,0)}=\beta^{\circ},\qquad \beta\in \Br_n^+,
\]
be the singularity knot.
Then \(PT:=T/\mathbb{G}_m\) acts freely on \(\Aug(\beta^{>},*_1,\cdots,*_n)\). Define the wild character variety \(\modulispace_{\Betti}(\beta)\) to be the good moduli space, in the sense of \cite{Alp13}, associated to \(\modulistack_1(\mathbb{P}^1,\{\infty\},(\Delta\beta\Delta)^{\circ})\). By \eqref{eq:augmentation_variety_vs_wild_character_stack},
\[
\modulispace_{\Betti}(\beta)
=
\Spec~\field[\Aug(\beta^{>},*_1,\cdots,*_n)]^T
=
\Aug(\beta^{>},*_1,\cdots,*_n)/PT.
\]
Then
\[
d:=\dim \modulispace_{\Betti}(\beta)
=
\dim \Aug(\beta^{>},*_1,\cdots,*_n) - \dim PT 
=
\tb(\beta^{>})+1
=
\mu
=
2\delta.
\]
By \eqref{eq:E-polynomial_vs_ruling_polynomial} and Remark~\ref{rem:ruling_polynomial_vs_HOMFLY-PT},
\begin{eqnarray*}
\widetilde R(\modulispace_{\Betti}(\beta);z)
&:=& q^{-d/2}E(\modulispace_{\Betti}(\beta);q) = q^{-d/2}|\modulispace_{\Betti}(\beta;\mathbb{F}_q)|\\
&=& zR(\modulistack_{\Betti}(\beta);z) = \widetilde R_{\beta^{>}}(z) = \coeff_{a^{\mu}}(P_{L_{(C,0)}}(a,z)).
\end{eqnarray*}
Now a wild \(P=W\) conjecture predicts that
\[
P_{\bullet}H^*(\overline{J}C)\cong W_{2\bullet}H^*(\modulispace_{\Betti}(\beta)).
\]
At the level of the \(t=-1\) specialization, this indeed holds by \eqref{eq:perverse_E-polynomial_vs_Hilbert_scheme_q-function} and \eqref{eq:Hilbert_scheme_q-function_vs_lowest_a-coefficient_of_HOMFLY-PT}:
\begin{eqnarray*}
q^{-\mu/2}P(\overline{J}C;q,-1)
&=&
q^{-\mu/2}(1-q)^2\sum_{\tau\ge 0}\chi(\Hilb^{[\tau]}(C))q^{\tau} \\
&=&
q^{-\mu/2}(1-q)\sum_{\tau\ge 0}\chi(\Hilb^{[\tau]}(C,0))q^{\tau} \\
&=&
\coeff_{a^{\mu}}(P_{L_{(C,0)}}(a,z)) \\
&=&
q^{-d/2}E(\modulispace_{\Betti}(\beta);q).
\end{eqnarray*}
Here, we use $\Hilb^{[\tau]}(C) = \sqcup_{0\leq \tau'\leq \tau} \Hilb^{[\tau']}(C,0)\times \mathbb{A}^{\tau-\tau'}$.
Finally, by definition (\ref{eqn:local_BPS_invariants}) of BPS invariants,
\[
\widetilde R(\modulispace_{\Betti}(\beta);z)
=
q^{-\mu/2}P(\overline{J}C;q,-1)
=
\sum_{h=0}^{\delta} n_h(C,0)\,z^{2h}.
\]
Therefore, Conjecture \ref{conj:log-concavity_for_local_BPS_invariants} can be viewed as a special case of Conjecture \ref{conj:log-concavity_of_E-polynomials_of_character_varieties}, realized by the wild character variety $\modulispace_{\Betti}(\beta)$ over $\mathbb{P}^1$ with one irregular singularity at $\infty$.

\section{Main evidence}

In addition to the empirical evidence discussed earlier, this section provides theoretical evidence for the conjectures stated above.

\subsection{Torus knots}

Let $(C,0)$ be the plane curve singularity defined by
\[
y^n-x^m=0,
\]
where $n,m$ are coprime positive integers. Then $\delta=\frac{(n-1)(m-1)}{2}$,
and the singularity link $L_{(C,0)}$ is the $(n,m)$-torus knot $T_{n,m} = ((\sigma_1\cdots\sigma_{n-1})^m)^{\circ}$.

\begin{theorem}\label{thm:log-concavity_for_torus_knots}
Conjecture~\ref{conj:log-concavity_for_local_BPS_invariants} holds for $(C,0)=\{y^n-x^m=0\}$ with $(n,m)=1$.
\end{theorem}

\begin{proof}
By the discussion in the previous section, it suffices to prove log-concavity for
$\widetilde{R}_{T_{n,m}}(z)\in \mathbb{Z}[z^2]$.

\noindent{}\textbf{Step 1.} By Jones \cite{Jon87}:
\begin{equation}\label{eq:Jones_formula_for_torus_knot}
P_{T_{n,m}}(a,z) = \frac{(1-q)(a/\sqrt{q})^{(n-1)(m-1)} \sum_{j=0}^{n-1} (-1)^j 
\frac{q^{jm+\frac{(n-j-1)(n-j)}{2}}}{[j]_q!\,[n-1-j]_q!} 
\prod_{i=-(n-1-j)}^{j} (q^i - a^2)}{(1-q^n)(1-a^2)},
\end{equation}
where
\[
[r]_q! := (1-q^r)[r-1]_q!,\qquad [0]_q! = 1;\qquad\qquad z=q^{1/2}-q^{-1/2}.
\]

\noindent{}\textbf{Step 2.} Using the $q$-binomial theorem
\[
\prod_{j=0}^{N-1}(1+q^j t) 
= 
\sum_{j=0}^{N} q^{\frac{j(j-1)}{2}} \genfrac{[}{]}{0pt}{}{N}{j}_{q} t^j,
\]
with appropriate substitutions, this yields
\begin{equation}\label{eq:torus_knot_q-binomial}
\widetilde{R}_{T_{n,m}}(z) 
= 
q^{-\frac{(n-1)(m-1)}{2}} \frac{\genfrac{[}{]}{0pt}{}{m+n}{n}_{q}}{\genfrac{[}{]}{0pt}{}{m+n}{1}_{q}},
\qquad \genfrac{[}{]}{0pt}{}{a}{b}_{q} := \frac{[a]_q!}{[b]_q!\,[a-b]_q!}.
\end{equation}

\noindent{}\textbf{Step 3.} From \eqref{eq:torus_knot_q-binomial}, we observe that
\begin{equation}\label{eq:torus_knot_factorization}
\widetilde{R}_{T_{n,m}}(z) 
= 
\prod_{j=1}^{\frac{(n-1)(m-1)}{2}} 
(q^{1/2} - \xi_j q^{-1/2})(q^{1/2} - \xi_j^{-1} q^{-1/2})
= 
\prod_{j=1}^{\frac{(n-1)(m-1)}{2}} \bigl(z^2 + 2 - \xi_j - \xi_j^{-1}\bigr),
\end{equation}
where each $\xi_j$ is a root of unity. Each factor $z^2 + 2 - \xi_j - \xi_j^{-1}$ is a polynomial in $z^2$ with nonnegative coefficients and is log-concave. By the multiplicativity of log-concavity, the product $\widetilde{R}_{T_{n,m}}(z)$ is log-concave with no internal zeros.
\end{proof}

\subsection{ADE singularities}

Recall the ADE planar curve singularities:
\begin{itemize}[wide,labelwidth=!,labelindent=0pt,itemindent=!]
\item \(A_n\;(n\ge 1)\): \(y^2+x^{n+1}=0\);
\item \(D_n\;(n\ge 4)\): \(xy^2+x^{n-1}=0\);
\item \(E_6\): \(y^3+x^4=0\);\quad \(E_7\): \(y^3+yx^3=0\);\quad \(E_8\): \(y^3+x^5=0\).
\end{itemize}

%The BPS invariants for these singularities were computed by Shende \cite[\S 5]{She12}.

\begin{proposition}[{\cite[\S 5]{She12}}]\label{prop:BPS_invariants_for_ADE}
The BPS invariants of ADE singularities are as follows: 
\begin{align}
A_{2\delta-1}&:\; n_h = \binom{\delta+h}{\delta-h}, \label{eq:BPS_A_odd}\\[4pt]
A_{2\delta}&:\; n_h = \binom{\delta+h+1}{\delta-h}, \label{eq:BPS_A_even}\\[4pt]
D_{2\delta-1}&:\; n_h = \binom{\delta+h-2}{\delta-h} + 2\binom{\delta+h-2}{\delta-h-1} + \binom{\delta+h-1}{\delta-h-2}, \label{eq:BPS_D_odd}\\[4pt]
D_{2\delta-2}&:\; n_h = \binom{\delta+h-3}{\delta-h} + 2\binom{\delta+h-3}{\delta-h-1} + \binom{\delta+h-2}{\delta-h-2}, \label{eq:BPS_D_even}\\[4pt]
E_6&:\; (n_0,\ldots,n_3) = (5,10,6,1), \label{eq:BPS_E6}\\
E_7&:\; (n_0,\ldots,n_4) = (2,11,15,7,1), \label{eq:BPS_E7}\\
E_8&:\; (n_0,\ldots,n_4) = (7,21,21,8,1). \label{eq:BPS_E8}
\end{align}
Equivalently, for each $\Gamma \in \{A_n,\ n\ge 1;\ D_n,\ n\ge 4;\ E_6;\ E_7;\ E_8\}$ with $\delta$-invariant $\delta(\Gamma)$, define
\[
m_k(\Gamma):=\#\{\text{independent sets of size }k\text{ in the Dynkin diagram }\Gamma\}.
\]
Then
\[
n_h(\Gamma)=m_{\delta(\Gamma)-h}(\Gamma).
\]
\end{proposition}

\begin{remark}\label{rem:BPS_invariants_vs_independence_polynomial}
For each ADE singularity $\Gamma$, Proposition~\ref{prop:BPS_invariants_for_ADE} shows that the sequence $(n_{\delta(\Gamma)-h})$ coincides with the sequence $(m_k(\Gamma))$ of coefficients of the \textbf{independence polynomial}
\[
I(\Gamma;x):= \sum_{k\ge 0} m_k(\Gamma)\, x^k
\]
of the Dynkin diagram $\Gamma$. 
This identification is a special feature of ADE singularities and does not seem to extend to the general setting of our conjectures.

In particular, independence polynomials of trees are not log-concave in general for sufficiently large order (see \cite{KL25}). This does not contradict our results, since the sequences considered here arise from geometric invariants rather than arbitrary independence polynomials.
\end{remark}

\begin{proof}
The statement was originally proved in \cite[\S 5]{She12} by computing the Euler characteristics of the first $\delta$ punctual Hilbert schemes.
We give an alternative proof using ruling polynomials.

By \cite[Ex.~2.5]{Cas22}, the Legendrian links associated to ADE singularities are:
\begin{itemize}
\item $A_n$: $(\sigma_1^{n+1})^{>}$;
\item $D_n$: $(\sigma_1^{n-2}\sigma_2\sigma_1^2\sigma_2)^{>}$;
\item $E_n$ ($n=6,7,8$): $(\sigma_1^{n-3}\sigma_2\sigma_1^3\sigma_2)^{>}$.
\end{itemize}

For $\Gamma \in \{A_n, D_n, E_6, E_7, E_8\}$, define
\[
N_\Gamma(z) := \sum_k m_k(\Gamma)\, z^{2(\delta(\Gamma)-k)}.
\]
By abuse of notation, we also denote by $\Gamma$ the associated Legendrian link. By
(\ref{eqn:BPS_invariants_vs_lowest_HOMFLY-PT}) and
(\ref{eqn:BPS_invariants_vs_ruling_polynomials}), it suffices to prove
\[
\widetilde{R}_\Gamma(z)=N_\Gamma(z),
\]
where $\widetilde{R}_\Gamma(z)=z^{b(\Gamma)}R_\Gamma(z)$ and $b(\Gamma)$ denotes the number of connected components of $\Gamma$.

Recall from \cite[Def.~2.3]{HR15} that
\[
R_\Gamma(z) = \sum_{\rho} z^{-\chi(\rho)},
\]
where $\rho$ runs over all $\mathbb{Z}/2$-graded normal rulings of $\Gamma$, and
\[
-\chi(\rho) = |S(\rho)| - \#\{\text{right cusps}\},
\]
where \(S(\rho)\) denotes the set of switches of \(\rho\), i.e. the crossings at which the ruling replaces the crossing by a pair of parallel strands.

%\vspace{-0.5cm}
\begin{figure}[!htbp]
\begin{center}
\begin{tikzpicture}[baseline=-0.5ex,scale=0.5]

% ------------------------------------------------------------
% parameters
% ------------------------------------------------------------
\def\XL{-1.2}      % left start of rainbow arcs
\def\X0{3.0}       % start of braid block
\def\Xb{9.0}       % end of twist box
\def\Xone{10}    % b_1 crossing center
\def\Xtwo{12.0}    % b_2 crossing center
\def\Xthree{14}  % b_3 crossing center
\def\Xfour{16}   % b_4 crossing center
\def\XR{17.5}      % right end of braid block
\pgfmathsetmacro{\XM}{(\X0+\XR)/2} % midpoint for rainbow arcs

% y-levels of the three strands
\def\ybot{-1.0}
\def\ymid{0.0}
\def\ytop{1.0}

% ------------------------------------------------------------
% styles
% ------------------------------------------------------------
\tikzset{
  lab/.style={font=\small},
  crosslab/.style={font=\scriptsize},
  boxlab/.style={font=\small}
}

% ------------------------------------------------------------
% left incoming arcs
% ------------------------------------------------------------
\draw[thick] (\XL,1.2) to[out=0,in=180] (\X0,\ybot);
\draw[thick] ({\XL+1.7},1.2) to[out=0,in=180] (\X0,\ymid);
\draw[thick] ({\XL+3.4},1.2) to[out=0,in=180] (\X0,\ytop);

% ------------------------------------------------------------
% crossing macro sigma_1 (bottom/middle)
% center at x = #1, label = #2
% ------------------------------------------------------------
\newcommand{\sigonecross}[2]{%
  \draw[thick] ({#1-1.0},\ytop) -- ({#1+1.0},\ytop);
  \draw[thick] ({#1-1.0},\ybot) to[out=0,in=180] ({#1+1.0},\ymid);
  \draw[white,fill=white] (#1,-0.5) circle (0.16);
  \draw[thick] ({#1-1.0},\ymid) to[out=0,in=180] ({#1+1.0},\ybot);
  \node[crosslab] at (#1,-1.1) {$#2$};
}

% ------------------------------------------------------------
% crossing macro sigma_2 (middle/top)
% center at x = #1, label = #2
% ------------------------------------------------------------
\newcommand{\sigtwocross}[2]{%
  \draw[thick] ({#1-1.0},\ybot) -- ({#1+1.0},\ybot);
  \draw[thick] ({#1-1.0},\ymid) to[out=0,in=180] ({#1+1.0},\ytop);
  \draw[white,fill=white] (#1,0.5) circle (0.16);
  \draw[thick] ({#1-1.0},\ytop) to[out=0,in=180] ({#1+1.0},\ymid);
  \node[crosslab] at (#1,-0.05) {$#2$};
}

% ------------------------------------------------------------
% twist box for sigma_1^{n-2}
% sigma_1 = crossing of strands 1 and 2 (bottom/middle)
% ------------------------------------------------------------
\sigonecross{4}{a_1}
\sigonecross{8}{a_{n-2}}
\draw[thick] (\X0+2,\ytop) -- (\Xb-2,\ytop);
\node[thick] at ({(\X0+\Xb)/2},-0.5) {$\cdots\cdots$};

\draw[thick,green] (\X0+0.2,-1.4) rectangle (\Xb-0.2,1.2);
\node[boxlab,green] at ({(\X0+\Xb)/2},-2) {$\sigma_1^{\,n-2}$};

% ------------------------------------------------------------
% four explicit crossings:
% b_1 = sigma_2, b_2 = sigma_1, b_3 = sigma_1, b_4 = sigma_2
% ------------------------------------------------------------
\sigtwocross{\Xone}{b_1}
\sigonecross{\Xtwo}{b_2}
\sigonecross{\Xthree}{b_3}
\sigtwocross{\Xfour}{b_4}

%-------------------------------------------------------------
% short exit segments to the right
% ------------------------------------------------------------
\draw[thick] ({\Xfour+1.0},\ytop) -- (\XR,\ytop);
\draw[thick] ({\Xfour+1.0},\ymid) -- (\XR,\ymid);
\draw[thick] ({\Xfour+1.0},\ybot) -- (\XR,\ybot);

% ------------------------------------------------------------
% right outgoing hooks
% ------------------------------------------------------------
\draw[thick] (\XR,\ytop)  to[out=0,in=180] ({\XR+0.8},1.2);
\draw[thick] (\XR,\ymid)  to[out=0,in=180] ({\XR+2.5},1.2);
\draw[thick] (\XR,\ybot)  to[out=0,in=180] ({\XR+4.2},1.2);

% ------------------------------------------------------------
% nested rainbow arcs
% ------------------------------------------------------------
\draw[thick]
({\XL+3.4},1.2) to[out=0,in=180] (\XM,2.4) to[out=0,in=180] ({\XR+0.8},1.2);

\draw[thick]
({\XL+1.7},1.2) to[out=0,in=180] (\XM,3.8) to[out=0,in=180] ({\XR+2.5},1.2);

\draw[thick]
(\XL,1.2) to[out=0,in=180] (\XM,5.2) to[out=0,in=180] ({\XR+4.2},1.2);

\end{tikzpicture}
\end{center}
%\vspace{-0.6cm}
\caption{The Legendrian link associated to the $D_n$-singularity: the rainbow closure $\bigl(\sigma_1^{\,n-2}\sigma_2\sigma_1^2\sigma_2\bigr)^>$.}
\label{fig:Legendrian_link_for_D_n-singularity}
\end{figure}
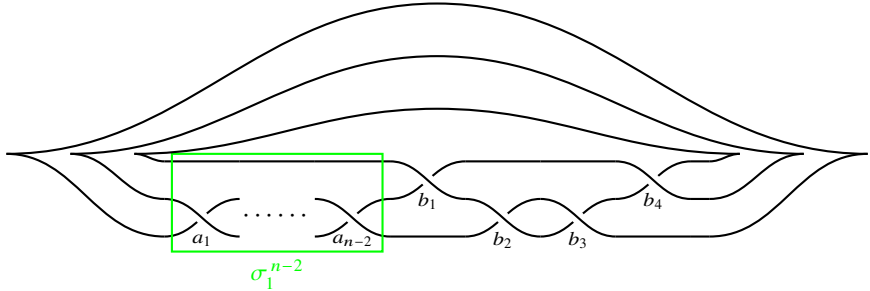

We treat the most nontrivial case $D_n$ (see Figure~\ref{fig:Legendrian_link_for_D_n-singularity}). 
Let $\rho$ be a $\mathbb{Z}/2$-graded normal ruling. 
From the figure, the two innermost cusps always bound the same eye. 
Consequently,
\begin{itemize}
\item either $b_1,b_4\in S(\rho)$;
\item or $b_1,b_4\notin S(\rho)$, in which case either both $b_2,b_3$ are switches or neither is.
\end{itemize}

This gives rise to three disjoint cases:

\begin{enumerate}[wide,labelwidth=!,labelindent=0pt,itemindent=!]

\item $b_1,b_4\in S(\rho)$.  
Replacing the crossings at $b_1,b_4$ by parallel strands and removing the innermost eye corresponds canonically to a ruling $\rho'$ of $(\sigma_1^n)^{>} = A_{n-1}$.  
Moreover,
\[
|S(\rho)| = |S(\rho')| + 2, \qquad -\chi(\rho) = -\chi(\rho') + 1.
\]

\item $b_1,b_4\notin S(\rho)$ and $b_2,b_3\in S(\rho)$.  
After resolving $b_2,b_3$ and removing the innermost eye, the remaining data corresponds canonically to a ruling $\rho'$ of $(\sigma_1^{n-2})^{>} = A_{n-3}$.  
Again,
\[
|S(\rho)| = |S(\rho')| + 2, \qquad -\chi(\rho) = -\chi(\rho') + 1.
\]

\item $b_1,b_2,b_3,b_4\notin S(\rho)$.  
Removing the innermost eye corresponds canonically to a ruling $\rho'$ of $(\sigma_1^{n-2})^{>} = A_{n-3}$ with
\[
|S(\rho)| = |S(\rho')|, \qquad -\chi(\rho) = -\chi(\rho') - 1.
\]

\end{enumerate}

Summing over all rulings, we obtain
\[
R_{D_n}(z) = z R_{A_{n-1}}(z) + z R_{A_{n-3}}(z) + z^{-1} R_{A_{n-3}}(z).
\]
Since $b(D_n) = b(A_{n-1}) + 1 = b(A_{n-3}) + 1$, this becomes
\begin{equation}\label{eqn:inductive_formula_for_ruling_polynomials_of_D_n}
\widetilde{R}_{D_n}(z)
= z^2 \widetilde{R}_{A_{n-1}}(z)
+ z^2 \widetilde{R}_{A_{n-3}}(z)
+ \widetilde{R}_{A_{n-3}}(z).
\end{equation}

\medskip

On the combinatorial side, label the Dynkin diagram of $D_n$ as
\begin{tikzpicture}[baseline=-0.5ex,scale=0.5]

\node at (0,0) {\small$1$};
\draw[thick,blue] (0.2,0)--(0.8,0);
\node at (1,0) {\small$2$};
\draw[thick,blue] (1.2,0)--(1.5,0);
\node at (2,0) {$\cdots$};
\draw[thick,blue] (2.5,0)--(2.8,0);
\node at (3.6,0) {\small$n-2$};
\draw[thick,blue] (4.4,0)--(4.7,0.3);
\node at (5.5,0.3) {\small$n-1$};
\draw[thick,blue] (4.4,0)--(4.7,-0.3);
\node at (5.3,-0.3) {\small$n$};

\end{tikzpicture}.
Recall that $m_k(D_n)$ is the number of independent sets of size $k$ in $D_n$.
We distinguish according to whether the terminal nodes $n$ and $n-1$ are included. 
This yields
\[
m_k(D_n)
= m_k(A_{n-1})
+ m_{k-1}(A_{n-3})
+ m_{k-2}(A_{n-3}).
\]
Equivalently,
\begin{equation}\label{eqn:inductive_formula_for_independence_polynomials_of_D_n}
N_{D_n}(z)
= z^2 N_{A_{n-1}}(z)
+ z^2 N_{A_{n-3}}(z)
+ N_{A_{n-3}}(z),
\end{equation}
since
\[
\delta(D_n)=\left\lfloor \tfrac{n}{2}\right\rfloor + 1
= \delta(A_{n-1})+1
= \delta(A_{n-3})+2.
\]

Comparing \eqref{eqn:inductive_formula_for_ruling_polynomials_of_D_n} and \eqref{eqn:inductive_formula_for_independence_polynomials_of_D_n}, we see that $\widetilde{R}_{D_n}(z)$ and $N_{D_n}(z)$ satisfy the same recursion. 
Since \(\widetilde{R}_{A_m}(z)=N_{A_m}(z)\) by a similar argument, or by \cite[Prop.~7.1]{Kal06}, it follows that
\[
\widetilde{R}_{D_n}(z)=N_{D_n}(z).
\]

The remaining cases are treated similarly or by direct verification.
\end{proof}

\begin{theorem}\label{thm:log-concavity_for_ADE}
Conjecture~\ref{conj:log-concavity_for_local_BPS_invariants} holds for all ADE singularities.
\end{theorem}

\begin{proof}
We treat each type separately.

\medskip
\noindent\textbf{Type \(A\) (odd):} \(A_{2\delta-1}\), \(\delta\ge 1\). Here
$n_h = \binom{\delta+h}{\delta-h}$.
The log-concavity can be checked directly from the identity
\[
\binom{\delta+h}{\delta-h}^2
-
\binom{\delta+h-1}{\delta-h+1}\binom{\delta+h+1}{\delta-h-1}
=
\frac{(2h+2)(\delta+h)!(\delta-h)!}{(\delta-h+1)!^2(\delta+h+1)}
\ge 0.
\]

\begin{remark}\label{rem:A_odd_factorization}
Alternatively, the log-concavity for type \(A_{2\delta-1}\) can also be deduced from a factorization as in Theorem~\ref{thm:log-concavity_for_torus_knots}. Set \(w=z^2\) and \(p_\delta(w):=\sum_{h=0}^{\delta}\binom{\delta+h}{2h}w^h\). Then \(p_\delta\) satisfies the Chebyshev-type recurrence
\[
p_{\delta+1}(w)=(2+w)\,p_\delta(w)-p_{\delta-1}(w),\qquad p_0=1,\; p_1=1+w.
\]
Substituting \(w=t-2\) gives \(p_\delta(t-2)=U_\delta(t/2)-U_{\delta-1}(t/2)\), where \(U_\delta\) is the Chebyshev polynomial of the second kind. From the explicit zeros of \(U_\delta-U_{\delta-1}\), one obtains
\[
p_\delta(w)=\prod_{j=0}^{\delta-1}\Bigl(w+4\sin^2\frac{(2j+1)\pi}{2(2\delta+1)}\Bigr).
\]
Each linear factor \(w+c_j\) with \(c_j>0\) is log-concave as a polynomial in~\(z\), and the multiplicativity of log-concavity yields the result.
\end{remark}

\medskip
\noindent\textbf{Type \(A\) (even):} \(A_{2\delta}\), \(\delta\ge 1\). Here
$n_h = \binom{\delta+h+1}{\delta-h}$.
The log-concavity follows from the torus-knot case (Theorem~\ref{thm:log-concavity_for_torus_knots}), since \(A_{2\delta}\) is the singularity
$y^2-x^{2\delta+1}=0$ (after replacing $x$ by $-x$), with link \(T_{2,2\delta+1}\).

\medskip
\noindent\textbf{Type \(E\):} \(E_6\), \(E_7\), \(E_8\). The log-concavity is verified directly from the explicit sequences in \eqref{eq:BPS_E6}--\eqref{eq:BPS_E8}:
\begin{itemize}[wide,labelwidth=!,labelindent=0pt,itemindent=!]
\item \(E_6\): \(10^2=100\ge 5\cdot 6=30\), and \(6^2=36\ge 10\cdot 1=10\);
\item \(E_7\): \(11^2=121\ge 2\cdot 15=30\), \(15^2=225\ge 11\cdot 7=77\), and \(7^2=49\ge 15\cdot 1=15\);
\item \(E_8\): \(21^2=441\ge 7\cdot 21=147\), \(21^2=441\ge 21\cdot 8=168\), and \(8^2=64\ge 21\cdot 1=21\).
\end{itemize}

\medskip
\noindent\textbf{Type \(D\):} \(D_n\), \(n\ge 4\). Write
\[
\delta = \left\lfloor \frac{n}{2}\right\rfloor + 1,\qquad k=\delta-h.
\]
Then by Proposition~\ref{prop:BPS_invariants_for_ADE},
\begin{equation}\label{eq:m_k_D_n}
m_k = n_h
=
\binom{n-k}{k} + \binom{n-k-1}{k-1} + \binom{n-k}{k-2},
\end{equation}
where \(\binom{a}{b}=0\) for \(b<0\) or \(b>a\). 
Recall that \(m_k\) is the coefficient of \(x^k\) in \(I(D_n;x)\) (see Remark~\ref{rem:BPS_invariants_vs_independence_polynomial}).

After simplification, we have
\begin{equation}\label{eq:m_k_D_n_simplified}
m_k = \frac{(n-k-1)!}{k!\,(n-2k+2)!}\,f(n,k),
\end{equation}
where
\begin{equation}\label{eq:f_n_k}
f(n,k) := n(n-2k+1)(n-2k+2) + k(k-1)(n-k).
\end{equation}
We must verify
\[
m_k^2 \ge m_{k-1}m_{k+1}
\]
for all admissible \(k\).

\textit{Trivial range.}
If \(k\le 0\), then \(m_{k-1}=0\), so the inequality is trivial. Similarly, if \(k\ge \frac{n+1}{2}\), then \(m_{k+1}=0\), so the inequality is again trivial. When \(n\) is even and \(k=\frac{n}{2}\), we have
\[
m_{k+1}=1,\qquad
m_k=\frac{k^2-k+4}{2},\qquad
m_{k-1}=\frac{k^4-2k^3+23k^2+2k}{24},
\]
and a direct computation verifies the inequality.

\textit{Reduction to a polynomial inequality.}
For \(k\le \frac{n-1}{2}\), the condition \(m_k^2 \ge m_{k-1}m_{k+1}\) can be rewritten, using \eqref{eq:m_k_D_n_simplified}, as
\begin{equation}\label{eq:D_n_sufficient_condition}
\frac{f(n,k-1)\,f(n,k+1)}{f(n,k)^2}
<
\frac{(k+1)(n-k-1)}{k(n-k)}
\cdot
\frac{(n-2k+4)(n-2k+3)}{(n-2k+2)(n-2k+1)}.
\end{equation}

Since \(n-2k-1\ge 0\), it suffices to prove
\begin{equation}\label{eq:D_n_sufficient_condition_simplified}
\frac{f(n,k-1)\,f(n,k+1)-f(n,k)^2}{f(n,k)^2}
<
\frac{n-2k-1}{k(n-k)} + \frac{4n-8k+10}{(n-2k+2)(n-2k+1)}.
\end{equation}

\textit{Explicit computation.}
Expanding \(f(n,k)\) as a polynomial in \(k\), we have
\[
f(n,k) = -k^3 + (5n+1)k^2 - (4n^2+7n)k + n^3 + 3n^2 + 2n.
\]
Set
\[
F(n,k):= f(n,k-1)\,f(n,k+1) - f(n,k)^2.
\]
A direct computation using the discrete Taylor expansion gives
\begin{equation}\label{eq:F_n_k}
F(n,k) = f(n,k)\,f''(n,k) - (f'(n,k))^2 + \tfrac{1}{4}(f''(n,k))^2 + 2f'(n,k) - 1,
\end{equation}
where \(f'\), \(f''\) denote derivatives with respect to \(k\). Hence
\begin{equation}\label{eq:F_n_k_explicit}
F(n,k)
=
-3k^{4}+(20n+4)k^{3}-(50n^{2}+20n-1)k^{2}+(34n^{3}+60n^{2}-8n-2)k-6n^{4}-24n^{3}-6n^{2}.
\end{equation}

\textit{Asymptotic analysis for large \(n\).}
Setting \(k=dn\) with \(0<d<1/2\), we have
\begin{eqnarray*}
f(n,k)
&=&
(1-4d+5d^{2}-d^{3})n^{3}+(3-7d+d^{2})n^{2}+2n \\
&=&
\left(\frac{1}{9}+\frac{9}{2}\left(\frac{4}{9}-d\right)^{2}+d^{2}\left(\frac{1}{2}-d\right)\right)n^{3}
+\left(-\frac{1}{4}+\left(\frac{1}{2}-d\right)\left(\frac{13}{2}-d\right)\right)n^{2}+2n \\
&\ge&
\frac{1}{9}n^{3}-\frac{1}{4}n^{2}+2n,
\end{eqnarray*}
and
\begin{eqnarray*}
&&F(n,k)\\
&=&
(-6+34d-50d^{2}+20d^{3}-3d^{4})n^{4}+(-24+60d-20d^{2}+4d^{3})n^{3}+(-6-8d+d^{2})n^{2}-2dn \\
&=&
\left(\frac{49}{40}-40\left(\frac{17}{40}-d\right)^{2}-20d^{2}\left(\frac{1}{2}-d\right)-3d^{4}\right)n^{4}
+\left(\frac{3}{2}-(51-18d)\left(\frac{1}{2}-d\right)-4d^{2}\left(\frac{1}{2}-d\right)\right)n^{3} \\
&&
+\left(-\frac{39}{4}-\left(\frac{1}{2}-d\right)\left(\frac{15}{2}-d\right)\right)n^{2}-2dn \\
&\le&
\frac{49}{40}n^{4} + \frac{3}{2}n^{3}-\frac{39}{4}n^{2}.
\end{eqnarray*}

Hence
\[
\frac{F(n,k)}{f(n,k)^2}
\leq
\frac{\frac{49}{40}n^{4} + \frac{3}{2}n^{3}-\frac{39}{4}n^{2}}
{\left(\frac{1}{9}n^3 - \frac{1}{4}n^2 + 2n\right)^2}.
\]

On the other hand, the right-hand side of \eqref{eq:D_n_sufficient_condition_simplified} satisfies
\[
\frac{n-2k-1}{k(n-k)} + \frac{4n-8k+10}{(n-2k+2)(n-2k+1)}
>
\begin{cases}
\dfrac{8n-16}{3n^{2}}+\dfrac{4n}{(n+2)(n+1)} & 0 < k \leq \dfrac{n}{4}, \\[6pt]
\dfrac{8n}{(n+4)(n+2)} & \dfrac{n}{4}< k \leq \dfrac{n-1}{2}.
\end{cases}
\]

When \(n>25\), in both cases we have
\[
\frac{n-2k-1}{k(n-k)} + \frac{4n-8k+10}{(n-2k+2)(n-2k+1)}
>
\frac{173}{30n},
\]
while
\[
\frac{F(n,k)}{f(n,k)^2}
\leq
\frac{\frac{49}{40}n^{4} + \frac{3}{2}n^{3}-\frac{39}{4}n^{2}}
{\left(\frac{1}{9}n^3 - \frac{1}{4}n^2 + 2n\right)^2}
<
\frac{265}{2n^{2}}.
\]
Since
\[
\frac{265}{2n^2} < \frac{173}{30n}
\qquad\text{for } n > \frac{265}{2}\cdot\frac{30}{173}\approx 22.98,
\]
the inequality \eqref{eq:D_n_sufficient_condition_simplified} holds for all \(n>25\).

\textit{Finite check.}
For $n\le 25$, the required log-concavity can be verified by direct computation.
This completes the proof for type \(D\).
\end{proof}

\subsection{A multiplicative property for ruling polynomials}

Our convention is to label $N$ parallel strands from bottom to top by $1,2,\cdots,N$.
Let $\beta_1,\beta_2\in \Br_n^+$, and $\gamma\in \Br_m^+$, and $N = n+m-1$.
Denote by $\tilde{\gamma}\in \Br_N^+$ the positive braid obtained from $\gamma$ by adding $n-1$ parallel strands from the bottom. In other words, if $\gamma = \sigma_{i_1}\cdots\sigma_{i_k}$ with $1\leq i_j\leq m-1$, then $\tilde{\gamma} = \sigma_{i_1+n-1}\cdots \sigma_{i_k+n-1}\in \Br_N^+$.
See Figure~\ref{fig:multiplicativity_for_ruling_polynomials} for an illustration.

%\vspace{-0.5cm}
\begin{figure}[!htbp]
\begin{center}
\begin{tikzpicture}[baseline=-0.5ex,scale=0.37]

% ------------------------------------------------------------
% parameters
% ------------------------------------------------------------
\def\XL{1.4}      % left start of rainbow arcs
\def\X0{9}       % start of braid block
\def\Xone{10}    % a_1 crossing center
\def\Xtwo{12.0}    % a_2 crossing center
\def\Xthree{14}  % a_3 crossing center
\def\Xfour{16}   % a_4 crossing center
\def\Xfive{18}   % a_5 crossing center
\def\Xsix{20}    % a_6 crossing center
\def\Xseven{22}    % a_7 crossing center
\def\Xeight{24}  % a_8 crossing center
\def\Xnine{26}   % a_9 crossing center
\def\Xten{28}   % a_{10} crossing center
\def\Xeleven{30}   % a_{11} crossing center
\def\Xtwelve{32}   % a_{12} crossing center
\def\XR{33}      % right end of braid block
\pgfmathsetmacro{\XM}{(\X0+\XR)/2} % midpoint for rainbow arcs

% y-levels of the three strands
\def\yone{-2.0}
\def\ytwo{-1.0}
\def\ythree{0}
\def\yfour{1.0}
\def\yfive{2.0}

% ------------------------------------------------------------
% styles
% ------------------------------------------------------------
\tikzset{
  lab/.style={font=\small},
  crosslab/.style={font=\scriptsize},
  boxlab/.style={font=\small}
}

% ------------------------------------------------------------
% left incoming arcs
% ------------------------------------------------------------
\draw[thick] (\XL,2.2) to[out=0,in=180] (\X0,\yone);
\draw[thick] ({\XL+1.7},2.2) to[out=0,in=180] (\X0,\ytwo);
\draw[thick] ({\XL+3.4},2.2) to[out=0,in=180] (\X0,\ythree);
\draw[thick] ({\XL+5.1},2.2) to[out=0,in=180] (\X0,\yfour);
\draw[thick] ({\XL+6.8},2.2) to[out=0,in=180] (\X0,\yfive);

% ------------------------------------------------------------
% crossing macro sigma_1
% center at x = #1, label = #2
% ------------------------------------------------------------
\newcommand{\sigonecross}[2]{%
  \draw[thick] ({#1-1.0},\ythree) -- ({#1+1.0},\ythree);
  \draw[thick] ({#1-1.0},\yfour) -- ({#1+1.0},\yfour);
  \draw[thick] ({#1-1.0},\yfive) -- ({#1+1.0},\yfive);
  \draw[thick] ({#1-1.0},\yone) to[out=0,in=180] ({#1+1.0},\ytwo);
  \draw[white,fill=white] (#1,-1.5) circle (0.16);
  \draw[thick] ({#1-1.0},\ytwo) to[out=0,in=180] ({#1+1.0},\yone);
  \node[crosslab] at (#1,-1.1) {$#2$};
}

% ------------------------------------------------------------
% crossing macro sigma_2
% center at x = #1, label = #2
% ------------------------------------------------------------
\newcommand{\sigtwocross}[2]{%
  \draw[thick] ({#1-1.0},\yone) -- ({#1+1.0},\yone);
  \draw[thick] ({#1-1.0},\yfour) -- ({#1+1.0},\yfour);
  \draw[thick] ({#1-1.0},\yfive) -- ({#1+1.0},\yfive);
  \draw[thick] ({#1-1.0},\ytwo) to[out=0,in=180] ({#1+1.0},\ythree);
  \draw[white,fill=white] (#1,-0.5) circle (0.16);
  \draw[thick] ({#1-1.0},\ythree) to[out=0,in=180] ({#1+1.0},\ytwo);
  \node[crosslab] at (#1,-0.05) {$#2$};
}

% ------------------------------------------------------------
% crossing macro sigma_3
% center at x = #1, label = #2
% ------------------------------------------------------------
\newcommand{\sigthreecross}[2]{%
  \draw[thick] ({#1-1.0},\yone) -- ({#1+1.0},\yone);
  \draw[thick] ({#1-1.0},\ytwo) -- ({#1+1.0},\ytwo);
  \draw[thick] ({#1-1.0},\yfive) -- ({#1+1.0},\yfive);
  \draw[thick] ({#1-1.0},\ythree) to[out=0,in=180] ({#1+1.0},\yfour);
  \draw[white,fill=white] (#1,0.5) circle (0.16);
  \draw[thick] ({#1-1.0},\yfour) to[out=0,in=180] ({#1+1.0},\ythree);
  \node[crosslab] at (#1,-0.05) {$#2$};
}

% ------------------------------------------------------------
% crossing macro sigma_4
% center at x = #1, label = #2
% ------------------------------------------------------------
\newcommand{\sigfourcross}[2]{%
  \draw[thick] ({#1-1.0},\yone) -- ({#1+1.0},\yone);
  \draw[thick] ({#1-1.0},\ytwo) -- ({#1+1.0},\ytwo);
  \draw[thick] ({#1-1.0},\ythree) -- ({#1+1.0},\ythree);
  \draw[thick] ({#1-1.0},\yfour) to[out=0,in=180] ({#1+1.0},\yfive);
  \draw[white,fill=white] (#1,1.5) circle (0.16);
  \draw[thick] ({#1-1.0},\yfive) to[out=0,in=180] ({#1+1.0},\yfour);
  \node[crosslab] at (#1,-0.05) {$#2$};
}

% ------------------------------------------------------------
% twelve explicit crossings:
% ------------------------------------------------------------
\sigonecross{\Xone}{}
\sigonecross{\Xtwo}{}
\sigtwocross{\Xthree}{}
\sigtwocross{\Xfour}{}
\sigthreecross{\Xfive}{}
\sigthreecross{\Xsix}{}
\sigfourcross{\Xseven}{}
\sigfourcross{\Xeight}{}
\sigthreecross{\Xnine}{}
\sigthreecross{\Xten}{}
\sigtwocross{\Xeleven}{}
\sigonecross{\Xtwelve}{}

% ------------------------------------------------------------
% right outgoing hooks
% ------------------------------------------------------------
\draw[thick] (\XR,\yfive)  to[out=0,in=180] ({\XR+0.8},2.2);
\draw[thick] (\XR,\yfour)  to[out=0,in=180] ({\XR+2.5},2.2);
\draw[thick] (\XR,\ythree)  to[out=0,in=180] ({\XR+4.2},2.2);
\draw[thick] (\XR,\ytwo)  to[out=0,in=180] ({\XR+5.9},2.2);
\draw[thick] (\XR,\yone)  to[out=0,in=180] ({\XR+7.6},2.2);

% ------------------------------------------------------------
% nested rainbow arcs
% ------------------------------------------------------------
\draw[thick]
({\XL+6.8},2.2) to[out=0,in=180] (\XM,3.4) to[out=0,in=180] ({\XR+0.8},2.2);

\draw[thick]
({\XL+5.1},2.2) to[out=0,in=180] (\XM,4.8) to[out=0,in=180] ({\XR+2.5},2.2);

\draw[thick]
({\XL+3.4},2.2) to[out=0,in=180] (\XM,6.2) to[out=0,in=180] ({\XR+4.2},2.2);

\draw[thick]
({\XL+1.7},2.2) to[out=0,in=180] (\XM,7.6) to[out=0,in=180] ({\XR+5.9},2.2);

\draw[thick]
(\XL,2.2) to[out=0,in=180] (\XM,9) to[out=0,in=180] ({\XR+7.6},2.2);

% ------------------------------------------------------------
% twist box for \beta_1, \tilde{\gamma}, \beta_2
% ------------------------------------------------------------
\draw[thick,green] (\Xone-0.3,-2.2) rectangle (\Xsix+0.3,1.2);
\node[boxlab,green] at ({(\Xone+\Xsix)/2},-2.7) {$\beta_1$};

\draw[thick,green] (\Xseven-0.3,0.8) rectangle (\Xeight+0.3,2.1);
\node[boxlab,green] at ({(\Xseven+\Xeight)/2},2.5) {$\gamma$};

\draw[thick,green] (\Xnine-0.3,-2.2) rectangle (\Xtwelve+0.3,1.2);
\node[boxlab,green] at ({(\Xnine+\Xtwelve)/2},-2.7) {$\beta_2$};

\end{tikzpicture}
\end{center}
%\vspace{-0.5cm}
\caption{An illustration for the rainbow closure $(\beta_1\tilde{\gamma}\beta_2)^{>}$: in the figure, $n=4$, $m=2$, $\beta_1 = \sigma_1^2\sigma_2^2\sigma_3^2$, $\beta_2 = \sigma_3^2\sigma_2\sigma_1$ $\in \Br_4^+$, $\gamma = \sigma_1^2\in \Br_2^+$, hence $\tilde{\gamma} = \sigma_4^2\in \Br_5^+$.}
\label{fig:multiplicativity_for_ruling_polynomials}
\end{figure}
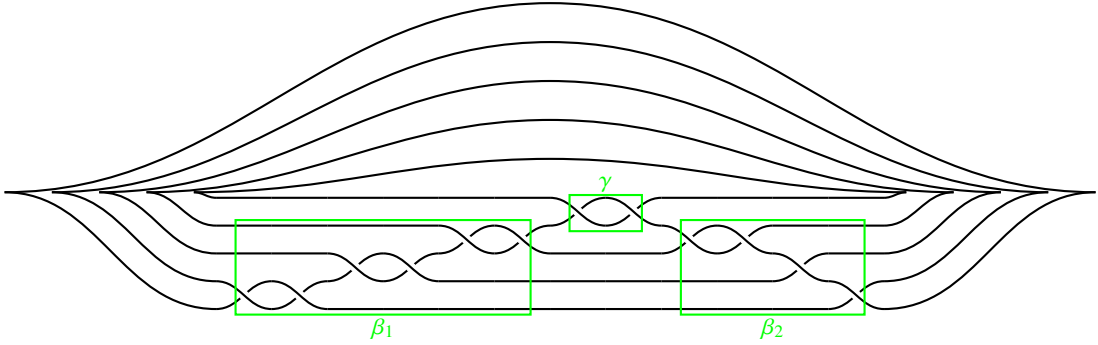

\begin{proposition}[Multiplicativity for ruling polynomials]\label{prop:multiplicativity_for_ruling_polynomials}
We have 
\[
R_{(\beta_1\tilde{\gamma}\beta_2)^{>}}(z) = zR_{\gamma^>}(z)R_{(\beta_1\beta_2)^{>}}(z).
\]
In particular, if Conjecture~\ref{conj:log-concavity_of_ruling_polynomials} holds for $\gamma^{>}$ and $(\beta_1\beta_2)^{>}$, then it also holds for $(\beta_1\tilde{\gamma}\beta_2)^{>}$.
\end{proposition}
As an application, Example~\ref{ex:computation_of_ruling_polynomial} can be computed easily using Proposition~\ref{prop:multiplicativity_for_ruling_polynomials}.

\begin{proof}
See Figure~\ref{fig:multiplicativity_for_ruling_polynomials} for an illustration.
In any normal ruling $\rho$ of $(\beta_1\tilde{\gamma}\beta_2)^{>}$, 
for each $i$, the pair of $i$-th innermost cusps bound the same eye.
It follows that the $(m-1)$ innermost eyes of $\rho$ determine a unique normal ruling $\rho_1$ of $\gamma^{>}$. 
Moreover, after resolving the switches of $\rho_1$ (which are supported inside the $\gamma$-region) and removing these $(m-1)$ innermost eyes, the remaining data of $\rho$ amounts to a normal ruling $\rho_2$ of $(\beta_1\beta_2)^{>}$.
In particular, $-\chi(\rho) = |S(\rho)| - N = (|S(\rho_1)| - m) + (|S(\rho_2)| - n) + 1 = -\chi(\rho_1) - \chi(\rho_2) + 1$. 
Conversely, any such $\rho_1,\rho_2$ gives rise to a normal ruling $\rho$.
Geometrically, this decomposition reflects the fact that the rainbow closure separates into two nested regions corresponding to $\gamma$ and $\beta_1\beta_2$, which interact through a single shared eye, producing the extra factor $z$.
Thus,
\[
R_{(\beta_1\tilde{\gamma}\beta_2)^{>}} = \sum_{\rho} z^{-\chi(\rho)} = \sum_{\rho_1,\rho_2} z z^{-\chi(\rho_1)}z^{-\chi(\rho_2)} = zR_{\gamma^>}(z)R_{(\beta_1\beta_2)^{>}}(z),
\]
where $\rho$, $\rho_1$, $\rho_2$ run over all $\mathbb{Z}/2$-graded normal rulings of $(\beta_1\tilde{\gamma}\beta_2)^{>}$, $\gamma^>$, $(\beta_1\beta_2)^>$ respectively.
The final statement follows immediately from the multiplicativity of log-concavity.
\end{proof}

\begin{corollary}
If $\beta = \sigma_{i_1}^{e_1}\cdots \sigma_{i_M}^{e_M} \in \Br_N^+$, and there exists $1\leq k\leq M$ such that 
$i_1<i_2<\cdots < i_k > i_{k+1} > \cdots > i_M$,
then Conjecture~\ref{conj:log-concavity_of_ruling_polynomials} holds for $\beta^{>}$.
\end{corollary}

\begin{proof}
By Theorem \ref{thm:log-concavity_for_ADE}, Conjecture~\ref{conj:log-concavity_of_ruling_polynomials} holds for the rainbow closures $(\sigma_1^n)^{>}$ of all $2$-strand positive braids. Now, apply Proposition~\ref{prop:multiplicativity_for_ruling_polynomials} inductively along the unique peak in the index sequence.
\end{proof}

\subsection*{Acknowledgements}
\addtocontents{toc}{\protect\setcounter{tocdepth}{1}}

We thank the following people and institutions for their hospitality, where parts of this work were presented and developed: Yu Pan (Center for Applied Mathematics, Tianjin University); Jun Zhang and Yongqiang Liu (Institute of Geometry and Physics, USTC); Jie Zhou and Dingxin Zhang (YMSC); and Michael McBreen and Conan Leung (IMS, CUHK).
We also thank Botong Wang and Laurentiu Maxim for helpful conversations.
The first author is grateful to Vivek Shende, David Nadler, and Lenhard Ng for their continued support and encouragement.

\addtocontents{toc}{\protect\setcounter{tocdepth}{2}}

\end{document}